\documentclass[10pt]{amsart}
\usepackage{amsmath,amsfonts,amssymb}
\usepackage[colorlinks=true, linkcolor=blue,
  citecolor=blue, urlcolor=blue]{hyperref}
\usepackage{ytableau, varwidth}
\usepackage{subcaption}
\usepackage{mathtools}
\usepackage{float}
\usepackage[alphabetic]{amsrefs}
\usepackage{tikz}
\usepackage{enumitem}

\def\newthm#1#2{\newtheorem{#1}[dummy]{#2}%
  \expandafter\def\csname#2\endcsname##1{\hyperref[#1:##1]{#2~\ref*{#1:##1}}}}
\newthm{thm}{Theorem}
\newthm{lemma}{Lemma}
\newthm{prop}{Proposition}
\newthm{cor}{Corollary}
\newthm{conj}{Conjecture}
\newthm{cons}{Consequence}
\newthm{fact}{Fact}
\newthm{obs}{Observation}
\newthm{guess}{Guess}
\theoremstyle{definition}
\newthm{defn}{Definition}
\newthm{remark}{Remark}
\newthm{example}{Example}
\newthm{question}{Question}
\newthm{notation}{Notation}

\makeatletter
\def\namedlabel#1#2{\begingroup
  #2%
  \def\@currentlabel{#2}%
  \phantomsection\label{#1}\endgroup
}
\makeatother

\newcommand{\Section}[1]{\hyperref[sec:#1]{Section~\ref*{sec:#1}}}
\newcommand{\Table}[1]{\hyperref[tab:#1]{Table~\ref*{tab:#1}}}
\newcommand{\eqn}[1]{\hyperref[eqn:#1]{(\ref*{eqn:#1})}}
\newcommand{\Figure}[1]{\hyperref[fig:#1]{Figure~\ref*{fig:#1}}}

\DeclareMathOperator{\Gr}{Gr}
\DeclareMathOperator{\Fl}{Fl}

\DeclareMathOperator{\Span}{Span}

\DeclareMathOperator{\QH}{QH}
\DeclareMathOperator{\QK}{QK}

\DeclareMathOperator{\codim}{codim}

\DeclareMathOperator{\opp}{{opp}}
\DeclareMathOperator{\comin}{{comin}}

\DeclareMathOperator{\ch}{\ch}

\usepackage{bm}

\renewcommand{\P}{{\mathbb P}}
\newcommand{\C}{{\mathbb C}}

\newcommand{\Z}{{\mathbb Z}}

\newcommand{\cM}{{\mathcal M}}

\DeclareMathOperator{\ev}{ev}
\newcommand{\wt}{\widetilde}

\newcommand{\wb}{\overline}
\newcommand{\ov}{\overline}

\newcommand{\ignore}[1]{}

\newcommand{\Mb}{\wb{\mathcal M}}

\begin{document}

\title{Curve neighborhoods of Seidel products in quantum cohomology}

\date{\today}

\author{Mihail Țarigradschi}
\address{Department of Mathematics, Rutgers University, 110
    Frelinghuysen Road, Piscataway, NJ 08854, USA}
\email{mt994@math.rutgers.edu}

\thanks{The author was partially supported by NSF grant DMS-2152316.}

\begin{abstract}
    A conjecture of Buch--Chaput--Perrin asserts that the two-pointed curve neighborhood corresponding to a quantum product of Seidel type is an explicitly given Schubert variety. We prove this conjecture for flag varieties in type \(A\).
\end{abstract}

\maketitle

\section{Introduction}

Given a flag variety \(X = G/P_X\), the Seidel representation on the small quantum cohomology ring \(\QH(X)|_{q=1}\) specialized at \(q=1\) can be described combinatorially, see \cites{belkale_transformation_2004,chaput_affine_2009}.
In the recent paper \cite{buch_seidel_2023}, this action is extended to the quantum K-theory ring \(\QK(X)|_{q=1}\) when \(X\) is cominuscule, and it is found that similar identities hold. On \(\QH(X)\) it is realized as the action of a subgroup \(W^{\comin} \subset W\) through the identity
\begin{equation}
    \label{eqn:QHeq}
    [X^w] \star [X^u] = q^d [X^{wu}],
\end{equation}
where \(w\in W^{\comin}\), \(u\in W\), and \(d=d_{\min}(w, u)\in H_2(X,\Z)\) denotes the smallest degree of a rational curve connecting general translates of \(X^w\) and \(X^u\). The subgroup \(W^{\comin}\) consists of the identity and the minimal length representatives of the longest Weyl group element \(w_0\) in \(W^M = W/W_M\) for each cominuscule flag variety \(M = G/P_M\):
\[
    W^{\comin} = \{1\} \cup \{ (w_0)^M : M = G/P_M \text{ where \(M\) is a cominuscule variety} \},
\]
here \((w_0)^M \in W\) denotes the representative of \(w_0\) in \(W^M\).

Equation \eqn{QHeq} also implies the following identity:
\[
    [X^w] \star [X^u] = q^d [\Gamma_d(X_{w_0 w}, X^u)],
\]
where \(\Gamma_d(X_{w_0 w}, X^u)\) is the two-pointed curve neighborhood defined as the union of all stable curves of degree \(d\) that pass through \(X_{w_0 w}\) and \(X^u\).

This led the authors of \cite{buch_seidel_2023} to conjecture that \(\Gamma_d(X_{w_0 w}, X^u)\) is a translate of \(X^{wu}\):
\begin{conj}[\cite{buch_seidel_2023}, Conjecture 3.11]
    \label{conj:gen}
    Let \(X = G/P_X\) be any flag variety. For \(u \in W, w \in W^{\comin}\), and \( d = d_{\min}(w,u) \in H_2(X,\Z)\), we have
    \[
        \Gamma_d(X_{w_0 w}, X^u) = w^{-1}.X^{w u}.
    \]
\end{conj}

It is known that the conjecture holds in the following cases:
\begin{itemize}
    \item \(d=0\) (see \cite{buch_seidel_2023}*{Proposition 2.2});
    \item \(X\) is cominuscule and \(w = w_0^X\) (see \cite{buch_seidel_2023}*{Lemma 3.1});
    \item \(X = \Gr(k, n)\) and \(w=w_0^{\P^{n-1}}\) (see \cite{li_seidel_2022}*{Proposition 4.5}).
\end{itemize}

In this paper we consider the following specialization:
\begin{conj}
    \label{conj:spec}
    Let \(X = G/P_X\) be a flag variety where \(P_X\) is a \textbf{maximal} parabolic subgroup of \(G\). For \(u \in W, w \in W^{\comin}\), and \( d = d_{\min}(w,u) \in H_2(X,\Z)\), we have
    \[
        \Gamma_d(X_{w_0 w}, X^u) = w^{-1}.X^{w u}.
    \]
\end{conj}

We prove the following reduction theorem.
\begin{thm}
    \label{thm:implies}
    \Conjecture{gen} follows from \Conjecture{spec} for the same \(G\).
\end{thm}

And then prove \Conjecture{spec} for Grassmannians of type A to conclude that \Conjecture{gen} holds for flag varieties of type A. Our proof is inspired from a description of two-pointed curve neighborhoods in Grassmannians given in \cite{li_seidel_2022}*{Proposition 4.5}.

\subsection*{Acknowledgements} We thank Anders Buch for useful discussions and bringing to our attention \Conjecture{gen}.

\section{Notation and Preliminaries}\label{sec:prelim}

Let \(G\) be a complex semisimple linear algebraic group, let \(T \subset B\) be a maximal torus and a Borel subgroup of \(G\). Denote by \(B^{-}\) the Borel subgroup opposite to \(B\), by \(\Phi\) the set of roots of \(G\), by \(\Delta\subset \Phi\) the set of simple roots, and by \(W\) the Weyl group of \(G\).

Let \(P_X \supseteq B\) be a parabolic subgroup of \(G\). In the flag variety \(X = G/P_{X}\) we consider the Schubert varieties \(X_u = \ov{B u P_X/P_X}\) and \(X^u = \ov{B^{-} u P_X/P_X}\) where \(u \in W\). Let \(W_X\) be the Weyl group of \(P_X\), then we denote \(W^X \subset W\) to be the set of minimal length representatives of the cosets in \(W/W_X\). For an element \(u \in W\), denote by \(u^X\in W^X\) the representative of \(u W_X \in W/W_X\). Let \(\Delta_{P_X} \subseteq \Delta\) be the set of simple roots that define \(P_X\), i.e. \(\beta \in \Delta_{P_X} \iff s_\beta \in W_X\).

Recall the general construction of curve neighborhoods (described for example in \cite{buch_finiteness_2013}). Let \(\ov{\cM}_{0,3}(X,d)\) denote the moduli space of \(3\)-pointed, genus \(0\) stable maps to \(X\) of effective degree \(d \in H_2(X,\Z)\). We have the evaluation maps \(ev_i : \ov{\cM}_{0,3}(X,d) \to X\) where \(i\in\{1,2,3\}\) that correspond to taking the image of the \(i\)-th marked point. Given two opposite Schubert varieties \(X_u, X^v\), define the \textit{Gromov--Witten variety} \(M_d(X_u, X^v) = ev_1^{-1}(X_u) \cap ev_2^{-1}(X^v) \subset \ov{\cM}_{0,3}(X,d)\). It describes the coefficients in the quantum product of Schubert classes in \(\QH(X)\):
\begin{equation}
    \label{eqn:GW}
    [X_u] \star [X^v] = \sum_{d \geq 0} (ev_3)_*[M_d(X_u, X^v)] q^d.
\end{equation}
We denote \(\Gamma_d(X_u,X^v) = ev_3(M_d(X_u, X^v))\). The subvariety \(\Gamma_d(X_u,X^v) \subset X\) is called a two-pointed curve neighborhood and consists of the union of all stable curves of degree \(d\) that pass through \(X_u\) and \(X^v\).

\section{Proof of the reduction to maximal parabolics}

In this section, we prove \Theorem{implies}. The main observation is the fact that the intersection of Schubert varieties \(X^{u^Y}\) for appropriate ``smaller" flag varieties \(Y\) is the Schubert variety \(X^u\), this follows from a combinatorial fact from \cite{bjorner_combinatorics_2005}.

\begin{lemma}
    \label{lemma:induct}
    Let \(X = G/P_X\) be a flag variety, \(F\subseteq X\) a subvariety, \(w\in W\), and \(g\in G\). Let \(Y = G/P_Y, Z=G/P_Z\) be flag varieties such that \(P_Y \cap P_Z = P_X\). Denote by \(p : X \to Y\) and \(q : X \to Z\) the projection maps, assume that \(p(F) \subseteq g.Y^w\) and \(q(F) \subseteq g.Z^w\). Then
    \( F \subseteq g.X^w\).
\end{lemma}
\begin{proof}
    Using \cite{buch_seidel_2023}*{Lemma 2.1(e)} we have
    \[
        F \subseteq p^{-1}(p(F)) \subseteq p^{-1}(g. Y^w) = g.X^{w^Y}.\]
    By the same argument for \(Z\), we have \(F \subseteq g.X^{w^Z}\), so that \(F \subseteq g.(X^{w^Y} \cap X^{w^Z})\). By \cite{bjorner_combinatorics_2005}*{Theorem 2.6.1}, using \(\Delta_{P_Y} \cap \Delta_{P_Z} = \Delta_{P_X}\), the elements \(w^Y, w^Z\) have a well-defined join equal to \(w^X\) in the poset \(W^X\). In particular \(X^{w^Y} \cap X^{w^Z} = X^{w^X} = X^w\), and we are done.
\end{proof}

\begin{proof}[Proof of \Theorem{implies}]
    By \cite{buch_finiteness_2013}*{Corollary 3.3}, the two-pointed Gromov--Witten variety \( M_d(X_{w_0 w}, X^u) \subset \Mb_{0,3}(X, d) \) is irreducible, and so is \( \Gamma_d(X_{w_0w},X^u) = \ev_3(M_d(X_{w_0 w}, X^u)) \subset X\).
    In the graded ring \(\QH(X)\), by \cites{belkale_transformation_2004, chaput_affine_2009} we have the identity \([X_{w_0 w}] \star [X^u] = q^d [X^{wu}]\).
    On the other hand, from Equation \eqn{GW} we have
    \[
        [X_{w_0 w}] \star [X^u] = q^d (ev_3)_*[M_d(X_{w_0 w}, X^u)] = q^d \deg(ev_3) [\Gamma_d(X_{w_0 w}, X^u)] ,
    \]
    so that \(\deg(ev_3)=1\), \([X^{wu}] = [\Gamma_d(X_{w_0 w}, X^u)]\), and  \(\codim X^{wu} = \codim \Gamma_d(X_{w_0 w}, X^u)\). Since Schubert varieties are irreducible, we are left with proving the inclusion \(\Gamma_d(X_{w_0w},X^u) \subseteq w^{-1}. X^{w u}\).

    We prove it by induction on \(| \Delta \setminus \Delta_X |\).
    If \(|\Delta \setminus \Delta_X| = 1\), we are done by assumption of \Conjecture{spec}.

    Otherwise, consider parabolic subgroups \(P_Y, P_Z \supsetneq P_X\) such that \(P_Y \cap P_Z = P_X\). Let \(p: X \to Y=G/P_Y\), \(q: X \to Z = G/P_Z\) be the projections. Then \(p(\Gamma_d(X_{w_0w},X^u)) \subseteq \Gamma_{p_*(d)}(Y_{w_0w},Y^u)\). Furthermore, \(p_*(d_{\min}(w,u)) = d_{\min}(w,u) \in H_2(Y, \Z)\) (see \cite{buch_euler_2020}*{Section 2.3} or \cites{chaput_affine_2009,belkale_transformation_2004}). Applying induction to \(Y\), \(\Gamma_{p_*(d)}(Y_{w_0w},Y^u) \subseteq w^{-1}.Y^{wu}\), so that
    \[
        p(\Gamma_d(X_{w_0w}, X^u)) \subseteq \Gamma_{p_*(d)}(Y_{w_0w}, Y^u) \subseteq w^{-1}.Y^{wu}.
    \]
    Analogously, we have \(q(\Gamma_d(X_{w_0w}, X^u)) \subseteq w^{-1}.Z^{wu}\). Using \Lemma{induct} for \(F = \Gamma_d(X_{w_0w}, X^u)\), we are done.
\end{proof}

\section{Proof of the conjecture in type A}

In this section we assume that the group \(G\) is of type A. Identify \(\Delta\) with the integers \(\{1, \dots, n-1\}\) and \(W\) with the symmetric group \(S_n\).
When \(P_X\) is the maximal parabolic corresponding to the simple root \(k\), i.e. \(\Delta_{P_X} = \Delta \setminus \{k\}\), then \(X\) is isomorphic to the Grassmannian variety \(\Gr(k,n)\) of \(k\)-planes in \(\C^n\).

We have the isomorphism \(W^{\comin} \cong \Z/n\Z \), with a generator \(w \in S_n\) given by
\begin{equation*}
    w(t) = \begin{cases}
        n,   & \text{ if } t=1            \\
        t-1, & \text{ if } 2\leq t \leq n
    \end{cases}.
\end{equation*}
The element corresponding to the simple root \(i \in \{1, \dots, n-1\}\) is given by \(w^i \in S_n\).

When \(X = \Gr(k,n)\) is a Grassmannian, we can describe the two-pointed curve neighborhoods using the \textit{quantum equals classical} theorems (see \cite{buch_gromov-witten_2003}). For an effective degree \(0 \leq d \leq \min(k,n-k)\) we denote \(Z_d = \Fl(k-d, d, k+d; n)\) to be the variety of three-step flags of dimensions \((k-d, d, k+d)\) in \(\C^n\). Similarly, we denote \(Y_d = \Fl(k-d, k+d; n)\) to be the variety of two-step flags of dimensions \((k-d, k+d)\) in \(\C^n\). We then have the projection maps \(p_d : Z_d \to X\), \(q_d: Z_d \to Y_d\) and \(\Gamma_d(X_u,X^v) = p_d(Z_d(X_u,X^v))\) where \(Z_d(X_u, X^v) = q_d^{-1}( q_d(p_d^{-1}(X_u)) \cap q_d(p_d^{-1}(X^v))) \subset Z_d\). Analogous to Equation \eqn{GW},
\[
    [X_u] \star [X^v] = \sum_{d \geq 0} (q_d)_*[Z_d(X_u, X^v)] q^d.
\]

Recall that Grassmannian Schubert varieties can be equivalently indexed by a partition. Fix a basis \(e_1, \dots, e_n\) of \(\C^n\) such that \(B\) acts on \(\C^n\) by upper-triangular matrices. Consider the complete flag \(E_{\bullet} = E_1 \subset E_2 \subset \dots \subset E_n\) where \(E_i = \Span\{e_1, \dots, e_i \}\), then \(B.E_\bullet = E_\bullet\). Similarly, consider the opposite flag \(E^{\opp}_{\bullet} = E^{\opp}_1 \subset E^{\opp}_2 \subset \dots \subset E^{\opp}_n\) where \(E^{\opp}_i = \Span\{e_n, \dots, e_{n-i+1} \}\), then \(B^{-}.E_\bullet^{\opp} = E_\bullet^{\opp}\).

Given a partition \( \lambda = (\lambda_1, \dots, \lambda_k) \) such that \(n-k \geq \lambda_1 \geq \dots \geq \lambda_k \geq 0\), define the \(B\)-stable Schubert variety
\[
    X_\lambda = \{ \Sigma\in X: \dim\Sigma \cap E_{i+\lambda_{k-i+1}} \geq i\ \text{for } i=1, \dots, k \},
\]
of dimension \(\dim X_\lambda = |\lambda|\), and the \(B^{-}\)-stable Schubert variety
\[
    X^\lambda = \{ \Sigma\in X: \dim\Sigma \cap E^{\opp}_{n-k+i-\lambda_i}  \geq i\ \text{for } i=1, \dots, k \}
\]
of codimension \(\codim X^\lambda = |\lambda|\).

These varieties satisfy \(X_\lambda = X_w\) and \(X^\lambda = X^w\) where \(w \in W^X \subset S_n\) is the permutation of minimal length such that \( \lambda = (w(k)-k, w(k-1)-(k-1), \dots, w(1)-1)\). We will also think of \(\lambda\) as a Young diagram inside a \(k \times (n-k)\) rectangle such that row \(i\) has \(\lambda_i\) boxes.

\begin{thm}
    \label{thm:Grr}
    Let \(X = \Gr(k, n)\). For \(u \in W\), \(w \in W^{\comin}\), and \( d = d_{\min}(w,u) \in H_2(X, \Z)\), we have
    \[
        \Gamma_d(X_{w_0 w}, X^u) = w^{-1}.X^{w u}
    \]
\end{thm}
\begin{proof}
    If \(w = 1\), then \(d=0\) and \( \Gamma_d(X_{w_0 w}, X^u) = \Gamma_0(X_{w_0}, X^u) = X_{w_0} \cap X^u = X^u = w^{-1}. X^{wu}\). Assume \(w \neq 1\).

    Let \(\beta \in \Delta\) be a simple root that indexes \(w\). Using the duality isomorphism \(\Gr(k,n) \cong \Gr(n-k, n)\), we may assume \( \beta \geq k\).
    As an element of \(S_n\), \( w_0 w = (\beta\ (\beta-1)\ \dots\ 1\ n\ (n-1)\ \dots (\beta+1)) \) in one-line notation. The corresponding partition in \(X\) is given by \((\beta-k, \beta-k, \dots, \beta-k) = (\beta-k)^k \). Denote by \(\lambda\) the partition associated to \(u^X\).

    We use \cites{fulton_quantum_2004} to get an explicit expression for \(d\). The degree \(d\) is the minimal degree of \(q\) in \([X_{(\beta-k)^k}] \star [X^\lambda] = [X^{(n-\beta)^k}] \star [X^\lambda]\). Consider the overlap of \(\lambda\) and the \(180^\circ\) rotation of \((n-\beta)^k\) in the \(k \times (n-k)\) rectangle. This overlap is a partition, call it \(\lambda^\prime\), where we remove the left-most \((\beta-k)\) columns from \(\lambda\) (see \Figure{degreefig}). Then \(d\) is the length of the longest NW-SE diagonal sequence of boxes in this overlap.
    Since \(\lambda^\prime\) is a partition, we can always move a NW-SE diagonal sequence of boxes to the NW corner of \(\lambda^\prime\). This gives the following formula for \(d\):
    \begin{equation}
        \label{eqn:degeqn}
        d = \max (\{0\} \cup \{j : \lambda_j - (\beta - k) \geq j\}).
    \end{equation}

    Here is an example where \(n=9, k=4, \lambda = (5,4,3,1), \beta=5\):
    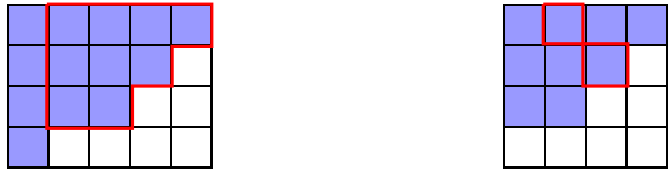
\begin{figure}[H]
        \begin{subfigure}{0.49\textwidth}
            \centering
            \ytableausetup{centertableaux}
            \begin{tikzpicture}[inner sep=0in,outer sep=0in]
                \node (n) {\begin{varwidth}{5cm}{
                            \begin{ytableau}
                                *(blue!40) \  & *(blue!40) \  & *(blue!40) \  & *(blue!40) \  &*(blue!40) \ \\
                                *(blue!40) \  & *(blue!40) \  & *(blue!40) \  & *(blue!40) \  &\ \\
                                *(blue!40) \  & *(blue!40) \  & *(blue!40) \ &\ &\ \\
                                *(blue!40) \ & \ & \ &\ &\
                            \end{ytableau}
                        }\end{varwidth}};
                \draw[very thick,red] ([xshift=1.5em]n.north west)--(n.north east)--([yshift=1.5em]n.east)--([yshift=1.5em, xshift=-1.5em]n.east)--([xshift=-1.5em]n.east)--([xshift=-2*1.5em]n.east)--([xshift=-2*1.5em, yshift=1.5em]n.south east)--([xshift=1.5em, yshift=1.5em]n.south west)--([xshift=1.5em]n.north west);
            \end{tikzpicture}
            \caption{\(\lambda\) given by the shaded region. Contour denotes the overlap.}
        \end{subfigure}
        \begin{subfigure}{0.49\textwidth}
            \centering
            \ytableausetup{centertableaux}
            \begin{tikzpicture}[inner sep=0in,outer sep=0in]
                \node (n) {\begin{varwidth}{5cm}{
                            \begin{ytableau}
                                *(blue!40) \  & *(blue!40) \  & *(blue!40) \  &*(blue!40) \ \\
                                *(blue!40) \  & *(blue!40) \  & *(blue!40) \  &\ \\
                                *(blue!40) \  & *(blue!40) \ &\ &\ \\
                                \ & \ &\ &\
                            \end{ytableau}
                        }\end{varwidth}};
                \draw[very thick,red] ([xshift=1.5em]n.north west)--([xshift=3em]n.north west)--([xshift=3em]n.west)--([xshift=-1.5em]n.east)--([xshift=-1.5em, yshift=-1.5em]n.north east)--([xshift=1.5em, yshift=-1.5em]n.north west)--([xshift=1.5em]n.north west);
            \end{tikzpicture}
            \caption{\(\lambda^\prime\) given by the shaded region. Contour denotes a longest NW-SE diagonal.}
        \end{subfigure}
        \caption{Finding \(d\) from \(\lambda\) and \(\lambda^\prime\).}
        \label{fig:degreefig}
    \end{figure}

    By the quantum equals classical theorem, we have that \(\Gamma_d(X_{w_0 w}, X^u) = p_d(Z_d(X_{w_0 w}, X^u))\). Looking at the three-step flags in \(q_d^{-1} q_d p_d^{-1} (X^u) \subset Z_d\), we have
    \begin{align*}
        q_d^{-1} q_d p_d^{-1} (X^u)                                                               & = q_d^{-1} q_d p_d^{-1} (X^\lambda)                                                                                                                                                                                            \\
                                                                                                  & = \{
        V_{k-d} \leq V_k \leq V_{k+d}                                                           : &                                     & \exists \wt{V_k} \text{ such that }   V_{k-d} \leq \wt{V_k} \leq V_{k+d} \text{ and }                                                                                                    \\
                                                                                                  &                                     &                                                                                               & \dim \wt{V_{k}} \cap E^{\opp}_{n-k+i-\lambda_i} \geq i \text{ for } 1 \leq i \leq k\}    \\
                                                                                                  & \subseteq \{
        V_{k-d} \leq V_k \leq V_{k+d}                                                           : &                                     & \dim V_{k+d} \cap E^{\opp}_{n-k+i-\lambda_i} \geq i \text{ for } 1 \leq i \leq k \text{ and }                                                                                            \\
                                                                                                  & \                                   &                                                                                               & \dim V_{k-d} \cap E^{\opp}_{n-k+i-\lambda_i} \geq i-d \text{ for } d+1 \leq i \leq k \},
    \end{align*}
    where the last condition follows from \(\dim V_{k-d} \cap E^{\opp}_{n-k+i-\lambda_i} \geq \dim \wt{V_k} \cap E^{\opp}_{n-k+i-\lambda_i} - d \geq i-d\).
    By symmetry, for the special case of \(X_{w_0 w}\), we have
    \begin{align*}
        q_d^{-1} q_d p_d^{-1} (X_{w_0 w})                                                         & = q_d^{-1} q_d p_d^{-1} (X_{(\beta-k)^k})                                                                                                                                                                 \\
                                                                                                  & \subseteq \{
        V_{k-d} \leq V_k \leq V_{k+d}                                                           : &                                           & \dim V_{k+d} \cap E_{i+\beta-k} \geq i \text{ for } 1 \leq i \leq k \text{ and }                                                                              \\
                                                                                                  & \                                         &                                                                                  & \dim V_{k-d} \cap E_{i+\beta-k} \geq i-d \text{ for } d+1 \leq i \leq k \} \\
                                                                                                  & = \{
        V_{k-d} \leq V_k \leq V_{k+d}                                                           : &                                           & V_{k-d} \subset E_{\beta},\ \dim V_{k+d} \cap E_{\beta} \geq k \}.
    \end{align*}
    Let \( V_k \in \Gamma_d(X_{w_0 w}, X^u) = p_d(Z_d(X_{w_0 w}, X^u)) = p_d( q_d^{-1} q_d p_d^{-1} (X_{w_0 w}) \cap q_d^{-1} q_d p_d^{-1}(X^u) ) \).
    If \(1 \leq i \leq d\), then \(k-i+\lambda_i + 1 > \beta\) by \eqn{degeqn}, so that \(E^{\opp}_{n-k+i-\lambda_i} \cap E_\beta = 0\) and
    \begin{align*}
        \dim V_k \cap (E^{\opp}_{n-k+i-\lambda_i} + E_\beta) & \geq \dim V_{k+d} \cap (E^{\opp}_{n-k+i-\lambda_i} + E_\beta) - d \\
                                                             & \geq
        \dim V_{k+d} \cap E^{\opp}_{n-k+i-\lambda_i} + \dim V_{k+d} \cap E_\beta -d \geq i + k-d.
    \end{align*}
    Otherwise, if \(d+1 \leq i \leq k\), then \(k-i+\lambda_i+1\leq \beta\) by \eqn{degeqn}, so that \(E^{\opp}_{n-k+i-\lambda_i} + E_\beta = \C^n\) and
    \begin{align*}
        \dim V_k \cap (E^{\opp}_{n-k+i-\lambda_i} \cap E_\beta) & \geq\dim V_{k-d} \cap (E^{\opp}_{n-k+i-\lambda_i} \cap E_\beta) \\
                                                                & = \dim V_{k-d} \cap E^{\opp}_{n-k+i-\lambda_i} \geq i-d.
    \end{align*}
    For simplicity, we denote
    \begin{equation}
        \label{eqn:defG}
        G_i = \begin{cases}
            E^{\opp}_{n-k+(i+d)-\lambda_{i+d}} \cap E_\beta,  & \text{if } 1\leq i\leq k-d,     \\
            E^{\opp}_{n-k+(i-k+d)-\lambda_{i-k+d}} + E_\beta, & \text{if } k-d+1 \leq i \leq k.
        \end{cases}
    \end{equation}
    So that we have a chain of inclusions
    \[
        G_1 \lneq G_2 \lneq \dots \lneq G_{k-d} \lneq G_{k-d+1} \lneq \dots \lneq G_k,
    \]
    and dimension conditions \( \dim V_k \cap G_{i} \geq i\) for all \(1 \leq i \leq k\).

    Note that the subspaces \(G_i\) are parts of the complete flag \(F^{\opp}_{\bullet}\) defined by
    \begin{equation*}
        F^{\opp}_i = \begin{cases}
            \Span \{e_\beta, e_{\beta-1}, \dots, e_{\beta-i+1}\},                        & \text{if } 1\leq i \leq \beta   \\
            \Span \{e_\beta, e_{\beta-1}, \dots, e_1, e_{n}, \dots, e_{n-(i-\beta)+1}\}, & \text{if } \beta+1\leq i \leq n
        \end{cases}
    \end{equation*}
    which is precisely given by \(F_{\bullet}^{\opp} = w^{-1}.E_{\bullet}^{\opp}\).

    The dimension conditions \(\dim V_k \cap G_i \geq i\) for \(1 \leq i \leq k\) define a Schubert variety with respect to the flag \(F_\bullet^{\opp}\), it is given by
    \begin{align*}
        \{V_k : \dim V_k \cap G_i \geq i \text{ for } 1\leq i\leq k \} & = w^{-1}.\{w.V_k : \dim V_k \cap G_i \geq i \text{ for } 1\leq i\leq k\}  \\
                                                                       & =  w^{-1}.\{V_k : \dim V_k \cap w.G_i \geq i \text{ for } 1\leq i\leq k\} \\
                                                                       & = w^{-1}.X^{v}
    \end{align*}
    for some \(v\in W^X\).

    We compute using \eqn{defG}
    \begin{align*}
        \ell(v) & = \codim w^{-1}. X^v = \sum (n-k+i-\dim G_i)                                                                          \\
                & = \sum_{1 \leq i \leq k-d} (n-k+i - (\beta - k+i+d-\lambda_{i+d}))                                                    \\
                & + \sum_{k-d<i \leq k} (n-k+i - ( \beta + n-k+(i-k+d)-\lambda_{i-k+d}))                                                \\
                & = \sum_{1 \leq i \leq k-d} (n - \beta - d + \lambda_{i+d} ) + \sum_{k-d < i \leq k} (k - \beta - d + \lambda_{i-k+d}) \\
                & = n(k-d) -\beta k + |\lambda| = n(k-d) -\beta k + \ell(u^X)
    \end{align*}
    On the other hand, from \eqn{QHeq} and \([X^w] = [X_{w_0 w}]\) we compute
    \begin{align*}
        \ell((wu)^X) & = \codim X^{wu} = \codim X_{w_0 w} + \codim X^u - d \deg(q) \\
                     & = \dim X - \dim X_{w_0 w} + \ell(u^X) - d n                 \\
                     & = k(n-k) - (\beta-k)k + \ell(u^X) - dn                      \\
                     & = n(k-d) -\beta k +\ell(u^X) = \ell(v)
    \end{align*}

    Since \(\Gamma_d(X_{w_0 w}, X^u) \subseteq w^{-1}.X^v\) and \(\dim X^v = \dim X^{wu} = \dim \Gamma_d(X_{w_0 w}, X^u)\) we get that the inclusion is an equality. Since \([\Gamma_d(X_{w_0 w}, X^u)] = [X^{wu}]\), we get \(v = (wu)^X\) from which the conclusion follows.
\end{proof}

By the above and \Theorem{implies}, we conclude that \Conjecture{gen} holds in type A.

\bibliography{bib.bib}
\end{document}